\documentclass{amsart}

\usepackage{amssymb}

\usepackage{graphicx}

\theoremstyle{plain}
\newtheorem{thm}{Theorem}[section]
\newtheorem{lem}[thm]{Lemma}

\newtheorem{cor}[thm]{Corollary}

\numberwithin{equation}{section}

\begin{document}

\title[Self-similar solutions to the curve shortening flow]{Self-similar solutions \\ to the curve shortening flow}


\author{Hoeskuldur P. Halldorsson}
\address{Department of Mathematics, Massachusetts Institute of Technology, 77 Massachu-setts Avenue, Cambridge, Massachusetts 02139-4307}
\email{hph@math.mit.edu}
\thanks{}

\subjclass[2010]{Primary 53C44}

\date{}

\dedicatory{}

\begin{abstract}
We give a classification of all self-similar solutions to the curve shortening flow in the plane.
\end{abstract}

\maketitle


\section{Introduction}
Some examples are known of immersed curves in the plane which move in a self-similar manner under the curve shortening flow (CSF). The two obvious ones are the straight line, which is not affected by the flow, and the circle, which shrinks homothetically to a point in finite time. The circle is the only simple closed curve which shrinks homothetically under the flow, but there also exist immersed closed curves with the same property. They were classified by Abresch and Langer in \cite{al}. A discussion about homothetically expanding curves can be found in \cite{ish} and \cite{urb}. The Grim Reaper curve is the graph of the function $-\log \cos (x)$. Calabi discovered that it translates upwards with constant speed under the flow. In \cite{alt}, Altschuler shows a picture of a yin-yang spiral and claims that it rotates under the flow. Some of these curves had earlier appeared in the physics literature \cite{mull}.

So we have curves which shrink, expand, rotate and translate under 
the flow. But are there any other self-similar solutions which combine 
two or more of these basic motions? The answer to that question is 
yes, and in this paper we find and classify all immersed curves in the plane which
move in a self-similar manner under the CSF. Our main results can be seen in the following list:

\begin{description}
\item[Translate] Only the Grim Reaper curve; Figure \ref{GrimReaper}.
\item[Expand] A one-dimensional family of curves. Each is properly embedded and asymptotic to the boundary of a cone;  Figures \ref{graf1}-\ref{graf2}.
\item[Shrink] A one-dimensional family of curves. Each is contained in an annulus and consists of identical excursions between its two boundaries; Figures \ref{AL1}-\ref{AL4}.
\item[Rotate] A one-dimensional family of curves. Each is properly embedded and spirals out to infinity; Figures \ref{A1B0samhverfi}-\ref{A1B0b3}.
\item[Rotate and expand] A two-dimensional family of curves. Each is properly embedded and spirals out to infinity; Figures \ref{A1B025samhverfi}-\ref{A1B025b3}.
\item[Rotate and shrink] A two-dimensional family of curves. Each has one end asymptotic to a circle, and the other is either asymptotic to the same circle or spirals out to infinity; Figures \ref{A1Bm005samhverfi}-\ref{A1Bm5halastjarnan}.
\end{description}

In Section \ref{allirfundnir}, we show how the problem of finding all self-similar solutions to the flow can be reduced to the study of a two-dimensional nonlinear system of ODEs. The system has two parameters, $A$ and $B$, which determine the type of the self-similar motion. The values of $A$ and $B$ are divided into four cases which yield curves with significantly different behaviour under the flow. In Sections \ref{primero}-\ref{ultimo} we look at each case separately and prove some properties of the corresponding curves by studying the solutions to the system of ODEs.

Section \ref{myndir} contains figures of the curves.

\section{All self-similar solutions found}
\label{allirfundnir}

Let $X: \mathbf R \rightarrow \mathbf C$ be an immersed curve in the plane. A self-similar motion of the curve is a complex valued function $\hat X$ on $\mathbf R \times  I$  of the form
\begin{equation}
\label{hreyfing}
\hat X(u,t) = g(t)e^{if(t)}X(u) + H(t), 
\end{equation}
where $I$ is an interval containing 0 and $f,g: I\rightarrow \mathbf R$ and $H: I \rightarrow \mathbf C$ are differentiable functions s.t.~ $f(0) = 0$, $g(0) = 1$ and $H(0) = 0$ and hence $\hat X(u,0) = X(u)$. The function $f$ determines the rotation, $g$ determines the scaling and $H$ is the translation term.

By definition, this motion is the curve shortening flow of the curve $X$ (up to tangential diffemorphisms) if and only if the equation
\begin{equation*}
\langle\frac{\partial \hat X}{\partial t}(u,t),N(u,t)\rangle = k(u,t)
\end{equation*}
holds for all $(u,t)\in \mathbf R \times I$. Here $N(u,t) = i\hspace{0.1cm} T(u,t)$ is the leftward pointing normal to the curve $u \mapsto \hat X(u,t)$ and the (signed) curvature $k$ is given by $k(u,t) = \langle \kappa(u,t),N(u,t)\rangle$, where $\kappa$ is the curvature vector of the curve. Straightforward calculations yield that this equation is equivalent to
\begin{equation}
\begin{aligned}
\label{allt}
g^2(t)f'(t)\langle X(u),T(u)\rangle &+ g(t)g'(t)\langle X(u),N(u)\rangle   \\
&+ g(t)\langle e^{-if(t)}H'(t),N(u)\rangle = k(u). 
\end{aligned}
\end{equation}
Now, we want this equation to hold for all $(u,t)\in \mathbf R \times I$. In particular, it must hold when $t=0$. Thus, the original curve $X$ has to satisfy the equation
\begin{equation}
\label{alltiupphafi}
A\langle X(u),T(u) \rangle + B \langle X(u),N(u) \rangle + \langle C, N(u) \rangle= k(u)
\end{equation}
for all $u\in \mathbf R$, where $A = f'(0)$, $B=g'(0)$ and $C=H'(0)$. \\

First we look at the case where the translation term vanishes so $C=0$. Assume $X$ is an immersed curve satisfying the equation
\begin{equation}
\label{jafnan}
A\langle X,T \rangle + B \langle X,N \rangle = k.
\end{equation}
Now, if we pick the functions $f$ and $g$ such that $g^2(t)f'(t)=A$ and $g(t)g'(t) = B$ for all $t\in I$, then it is clear that Equation \eqref{allt} holds for all $(u,t)\in \mathbf R \times I$. A valid choice for $f$ and $g$ is
\begin{equation*}
f(t) = \begin{cases} \frac{A}{2B}\log(2Bt+1) & \text{if } B\neq0,\\
At & \text{if } B=0, \end{cases} \quad \text{and} \quad g(t) = \sqrt{2Bt+1}.
\end{equation*}
Hence, under the CSF, the curve $X$ moves in the self-similar manner governed by $f$ and $g$ as shown in Equation \eqref{hreyfing}. In other words, it rotates around the origin (unless $A=0$) and dilates outwards if $B>0$ and inwards if $B<0$. This motion will be called screw-dilation.\\

What happens if we include the translation term $H$? Assume that the curve $X$ satisfies Equation \eqref{alltiupphafi} for all $u\in \mathbf R$. If $A$ and $B$ are not both $0$, we put $\tilde{X}(u)  = X(u) + \frac{C}{B+iA}$, and direct calculations yield that the curve $\tilde X$ satisfies Equation \eqref{jafnan}. But, as we saw before, that means  $\tilde X$ screw-dilates around 0 and hence $X$ screw-dilates around the point $-\frac{C}{B+iA}$. So when we have rotation or dilation (or both), introducing the translation term doesn't give us any new solutions.

However, if both $A$ and $B$ are 0 we can take $f=0$, $g=1$, $H=Ct$, and then it is clear that Equation \eqref{allt} is satisfied for all $(u,t)\in \mathbf R \times I$. Thus, the curve $X$ moves by translation in the direction $C$ under the CSF. Hence, the curve is either a straight line with $C$ as its direction vector (making the translation vacuous) or a part of the curve can be written as a graph of a function defined on a line perpendicular to the direction $C$. However, in \cite{eck}, it is shown that in this case the curve is the (possibly rescaled) graph of the function $x\mapsto -\log(\cos x)$, the so-called Grim Reaper curve; see Figure \ref{GrimReaper}.

So the problem has been reduced to finding all immersed curves $X:\mathbf R \rightarrow \mathbf C$ satisfying Equation \eqref{jafnan}. From now on we will parametrize the curves by arc length. For every curve $X$, we have
\begin{equation*}
\begin{aligned}
\frac{d}{ds}\langle X,T \rangle  &= 1 + k\langle X,N \rangle, \\
\frac{d}{ds}\langle X,N \rangle  &= - k\langle X,T \rangle.
\end{aligned}
\end{equation*}
Thus, if we define the functions
\begin{equation*}\begin{aligned}
x&=A\langle X,T \rangle + B \langle X,N \rangle,\\
y &=-B\langle X,T \rangle + A \langle X,N \rangle ,
\end{aligned}\end{equation*}
which can also be written as
\begin{equation*}
x+iy = (A-iB)(\langle X,T\rangle + i\langle X,N\rangle),
\end{equation*} then they satisfy
\begin{equation*}\begin{aligned} 
x' &=ky+A,\\
y' &=-kx-B.
\end{aligned}\end{equation*}
Note that we always have
\begin{equation*}\begin{aligned} 
X &=\langle X,T \rangle T + \langle X,N \rangle N\\
&=T( \langle X,T \rangle + i \langle X,N \rangle)\\
&=e^{i\theta} \frac{x+iy}{A-iB},
\end{aligned}\end{equation*}
where $\theta(s) = \int_0^s k(s')ds' + \theta_0$ and  $T(0) = e^{i\theta_0}$.

Now, we are looking for curves $X$ which satisfy the equation $x = k$, and by what was shown above, the following method yields all possible curves with that property. We let $x$ and $y$ be the unique solution to the system of differential equations
\begin{equation*}
\left\{\begin{aligned}
x' &=xy+A,\\
y' &=-x^2-B
\end{aligned}
\right.
\end{equation*}
for any initial values $x(0) = x_0$ and $y(0)=y_0$. Note that
\begin{equation*}
\frac{d}{ds}\sqrt{x^2+y^2} = \frac{Ax-By}{\sqrt{x^2+y^2}} \leq \sqrt{A^2+B^2},
\end{equation*}
so $x$ and $y$ can't blow up in finite time, and hence the solution is defined on all of $\mathbf R$.
Then we put  $\theta(s) = \int_0^s x(s')ds' + \theta_0$ for any $\theta_0$. Finally, we define the curve as
\begin{equation*}
\label{Xjafnan}
X = e^{i\theta} \frac{x+iy}{A-iB}.
\end{equation*}
Note that
\begin{equation*}\begin{aligned}
X' &=e^{i\theta}\frac{i\theta ' (x+iy) + x'+iy'}{A-iB}\\ &=e^{i\theta}\frac{ix (x+iy) + xy+A+i(-x^2-B)}{A-iB} \\ &=e^{i\theta},
\end{aligned}\end{equation*}
so $X$ is parametrized by arc length with tangent $T=e^{i\theta}$, and hence the curvature $k$ is equal to $\theta' = x$. Finally,
\begin{equation*}\begin{aligned}
A\langle X,T \rangle +B\langle X,N \rangle&=\langle X,(A+iB)T \rangle \\&=\text{Re}(X(A-iB)e^{-i\theta})\\ &=x\\&=k,
\end{aligned}\end{equation*}
so Equation \eqref{jafnan} is indeed satisfied. Hence, we have proved the following.

\begin{thm}
For each value of $A$ and $B$ there exists an immersed curve $X$ satisfying Equation \eqref{jafnan}.
\end{thm}

In fact, unless $A$ and $B$ are both 0, we have a whole family of immersed curves satisfying Equation \eqref{jafnan}. The rest of this paper will be devoted to classifying these curves. 

There is a lot of symmetry in the problem. Each curve is determined by the three parameters 
$x_0$, $y_0$ and $\theta_0$.  However, changing $\theta_0$ only rotates the curve, so in order to avoid getting identical curves we will from now on assume $\theta_0=0$.  Also note that rescaling the value of $(A,B)$ by $c$ corresponds to rescaling the curve $X$ by $c^{-\frac{1}{2}}$.

Define the functions $F(x,y) = xy+A$ and $G(x,y) = -x^2-B$ such that we have
\begin{equation*}\left\{\begin{aligned} 
x' &=F(x,y),\\
y' &=G(x,y).
\end{aligned}\right.\end{equation*}
Note that $F(-x,-y) = F(x,y)$ and $G(-x,-y) = G(x,y)$, and therefore $s \mapsto -(x(-s),y(-s))$ is also a solution to the system of ODEs, which of course just corresponds to parametrizing the curve $X$ backwards. This symmetry will be used to simplify some arguments.

It will turn out to be convenient to also work with the functions $\langle X,T\rangle$ and $\langle X,N\rangle$, which from here on will be denoted by $\tau$ and $\nu$, respectively. Recall that
\begin{equation*}
\label{xytaueta}
x+iy = (A-iB)(\tau+i\nu),
\end{equation*}
and since the curvature equals $x$, their derivatives are given by
\begin{equation*}\begin{aligned}
\tau' &=1+x\nu,\\
\nu' &=-x\tau.
\end{aligned}\end{equation*}
Hence, they are solutions to the system of differential equations
\begin{equation*}\left\{\begin{aligned}
\tau' &=1+A\tau\nu+B\nu^2,\\
\nu' &=-A\tau^2-BÐ\tau\nu,
\end{aligned}\right.\end{equation*}
and their initial values satisfy the equation $x_0+iy_0= (A-iB)(\tau_0+i\nu_0)$.

The curve $X$ can therefore be written as
\begin{equation}
\label{EnneinXjafnan}
X = e^{i\theta} \frac{x+iy}{A-iB} = e^{i\theta}(\tau+i\nu) = e^{i\phi}r,
\end{equation}
where
\begin{equation*}
r  = |X| = \sqrt{\tau^2+\nu^2} = \frac{\sqrt{x^2+y^2}}{\sqrt{A^2+B^2}} 
\end{equation*}
and
\begin{equation*}
\phi = \theta +\arg(\tau+i\nu) = \theta + \arg(x+iy) + \arg(A+iB).
\end{equation*}
\\

All possible values of $A$ and $B$ will be divided into 4 cases which yield curves with significantly different properties and behaviour under the flow. Those are:
\begin{itemize}
\item $A \neq 0$ and $B\geq 0$: Curves which rotate and expand.
\item $A \neq 0$ and $B< 0$: Curves which rotate and shrink.
\item $A=0$ and $B<0$: Shrinking curves.
\item $A=0$ and $B>0$: Expanding curves.
\end{itemize}
The case $A=B=0$ only yields straight lines and is therefore omitted. \\

Before we look at each case separately,  let's address a simple question: 
Can the curves have double points? Note that a double point would have to be transversal by the uniqueness of our system of ODEs.
Now, assume $X$ is any curve with a transversal double point. Then $X$ has a loop, $\Gamma$. Let the curve be parametrized such that the region enclosed by the loop is on the left side of the curve, and let $\alpha \in (0,2\pi)$ be the interior angle of the loop at the double point. If we let the curve $X$ flow under the CSF, then the time derivative at time $0$ of the area enclosed by the loop is
\begin{equation}
\label{flatarmal}
-\int_\Gamma k \hspace{0.1cm}ds = - (\pi+\alpha) < 0.
\end{equation}
Hence, the area enclosed by any loop is always strictly decreasing under the CSF. That means that if a curve expands under the flow (corresponding to $B \geq 0$), then it can't have any double points since the area enclosed by the corresponding loops would be increasing.

\section{Curves which rotate and expand}
\label{primero}

For each $A\neq0$ and $B\geq 0$, we have the following description of the curves satisfying Equation \eqref{jafnan}.
\begin{thm}
\label{snuaut}
The curves are properly embedded, have one point closest to the origin and consist of two arms coming out from this point which strictly go away from the origin to infinity. Each arm has infinite total curvature and spirals infinitely many circles around the origin. The curvature goes to $0$ along each arm, and their limiting growing direction is $B+iA$ times the location.

The curves form a one-dimensional family parametrized by their distance to the origin, which can take on any value in $[0,\infty)$.

If $B=0$, then under the CSF the curves rotate forever with constant angular speed $A$.

 If $B>0$, then under the CSF the curves rotate and expand forever with angular function $\frac{A}{2B}\log(2Bt+1)$ and scaling function $\sqrt{2Bt+1}$.
\end{thm}

Curves of this kind can be seen in Figures \ref{A1B0samhverfi}-\ref{A1B025b3}.\\

The proof will be the result of a series of lemmas.
By reflecting the curve if necessary, we can assume $A>0$. We start with a simple observation.

\begin{lem}
\label{fyrsta}
The curvature $x$ has at most one zero, is negative before it and positive after it. Moreover, $x$ has at most two extrema, a minimum below $0$ and a maximum above $0$.
\end{lem}
\begin{proof}
The first statement is true, since if $x(s) = 0$, then $x'(s) = A >0$, so $x$ has to grow through each of its zeros and thus can only have one. If $x'(s) =0$, then $x''(s) = -x(s)(x^2(s)+B)$, so all extrema above 0 are maxima and all extrema below 0 are minima. Therefore, $x$ has at most one minimum and at most one maximum which are below and above 0, respectively.
\end{proof}

\begin{lem}
\label{takm}
The curvature $x$ is bounded. 
\end{lem}
\begin{proof}
Can $x>0$, $x' >0$ hold for all $s \geq s_0$ for some $s_0$? No, since then $y' \leq -x^2(s_0)<0$, so $y$ decreases at least linearly, and hence $x' = xy+A$ will sooner or later reach 0. Similarly, $x>0$, $x'<0$ cannot hold for all $s\leq s_0$, since then $y' \leq -x^2(s_0)<0$, so $y$ increases at least linearly in the backwards direction, and hence $x' = xy+A$ will reach 0 for a small enough $s$. From this and Lemma \ref{fyrsta} we conclude that if $x$ reaches a positive value, then it has to have a global maximum. By symmetry it is clear that if $x$ reaches a negative value, then it has a global minimum.
\end{proof}

\begin{lem}
\label{ylimit}
 $\lim_{s\rightarrow \pm \infty}y = \mp \infty$.
 \end{lem}
 \begin{proof}
If $B>0$, then $y'\leq -B <0$, so the statement is obvious. Therefore we look at the case $B=0$.
Since $y' = -x^2 \leq 0$, $y$ is decreasing. Assume $y$ is bounded from below so that $\lim_{s\rightarrow \infty}y$ is finite. Now, $y''=-2x^2y-2Ax$ is bounded, since both $x$ and $y$ are bounded. Hence, by Barbalat's lemma (see \cite{sl}) $\lim_{s\rightarrow \infty}y'=0$. But $y' = -x^2$, so $\lim_{s\rightarrow \infty}x=0$. However, $x' = xy+A$, so $\lim_{s\rightarrow \infty}x'=A$, which contradicts $\lim_{s\rightarrow \infty}x=0$. Hence $\lim_{s\rightarrow \infty}y = -\infty$, and by symmetry, $\lim_{s\rightarrow -\infty}y = \infty$.
\end{proof}

\begin{lem}
\label{rlem}
$\lim_{s\rightarrow \pm \infty}r = \infty$ and $r$ has exactly one extremum, a global minimum.
\end{lem}
\begin{proof}
The first statement follows from Lemma \ref{ylimit}, since $r^2(A^2+B^2) = x^2+y^2$. To see the second one, note that $\frac{d}{ds}r^2 = 2\tau$, and if $\tau(s) = 0$, then $\tau'(s) = B\nu^2(s)+1 > 0$.
\end{proof}
This lemma implies that the curve $X$ has exactly one point closest to the origin. The curve consists of two arms coming out from this point, where each arm is strictly going away from the origin to infinity. Hence, each of these arms is properly embedded. Since we already know that $X$ doesn't have double points (because $B\geq 0$) we have the following interesting result.
\begin{cor} 
The curve $X$ is properly embedded.
\end{cor}

Our next goal is to show that the curvature $x$ goes to 0 along each arm.

\begin{lem}
\label{rod}
 The curvature $x$ has a negative global minimum, a positive global maximum and no other extrema.
\end{lem}
\begin{proof}
If $x<0$ for all $s$, then $x' = xy+A \geq A>0$ for all $s$ large enough (by Lemma \ref{ylimit}), and hence $x$ will eventually attain the value 0, a contradiction.
Similarly, $x>0$ can't hold for all $s$. Therefore, $x$ takes both negative and positive values, and the argument from the proof of  Lemma \ref{takm} shows that it has a global minimum and a global maximum. By Lemma \ref{fyrsta}, these are the only extrema.
\end{proof}

\begin{lem}
\label{etalimit}
$\lim_{s\rightarrow \pm \infty}\nu = \mp \infty$.
\end{lem}
\begin{proof}
Follows from Lemmas \ref{takm} and \ref{ylimit} since $\nu = \frac{Bx+Ay}{A^2+B^2}$.
\end{proof}

\begin{lem} $\lim_{s\rightarrow \pm \infty}x = 0$.
\label{xlimit}
\end{lem}
\begin{proof}
By Lemma \ref{rod}, we know that $x$ has a finite limit in each direction.
If the limit of $x$ when $s \rightarrow \infty$ is not 0, then $x>\delta>0$ for all $s\geq s_0$. By Lemma \ref{etalimit}, we can pick $s_0$ such that $\nu <-\frac{2}{\delta}$ for all $s\geq s_0$. But then $\tau' = 1+x\nu < 1-x\frac{2}{\delta} < -1$ for $s\geq s_0$, so $\lim_{s\rightarrow \infty}\tau = -\infty$. But $x = A\tau + B\nu$, so this yields $\lim_{s\rightarrow \infty}x = -\infty$, a contradiction. The other case follows by symmetry.
\end{proof}
Now we can determine the behaviour of the two arms far away from the origin. First let us state a basic limit result.

\begin{lem}
\label{markgildi}
$\lim_{s\rightarrow \pm \infty}\frac{\tau}{\nu} = -\frac{B}{A}$,\\
$\lim_{s\rightarrow \pm \infty}\frac{\tau}{r} = \pm\frac{B}{\sqrt{A^2+B^2}}$, \\
$\lim_{s\rightarrow \pm \infty}\frac{\nu}{r} = \pm\frac{-A}{\sqrt{A^2+B^2}}$.
\end{lem}
\begin{proof}
Follows from Lemmas  \ref{ylimit} and \ref{xlimit} since $\tau = \frac{Ax-By}{A^2+B^2}$ and $\nu = \frac{Bx+Ay}{A^2+B^2}$.
\end{proof}

By Equation \eqref{EnneinXjafnan} and the previous lemma, we immediately get the following corollary which shows the limiting growing direction of the arms.
\begin{cor}
\label{stefna}
$\lim_{s\rightarrow \pm \infty}\frac{rT}{X} = \pm \frac{B+iA}{\sqrt{A^2+B^2}}$.
\end{cor}

\begin{lem}
\label{philimit}
$\lim_{s\rightarrow \pm \infty}\phi = +\infty$.
\end{lem}
\begin{proof}
Easy calculations give $\frac{d\phi}{d \log(r)} = -\frac{\nu}{\tau}$ which goes to $\frac{A}{B}$ (or $+\infty$ when $B=0$) as $s \rightarrow \pm\infty$. Hence the result follows from Lemma \ref{rlem}.
\end{proof}

The above results show that each arm spirals infinitely many circles out from the origin and that its limiting growing direction is $B+iA$ times the location. Since $\phi = \theta + \arg(\tau+i\nu)$ and we just saw that $\arg(\tau+i\nu)$ has finite limits as $s \rightarrow \pm\infty$, we get the following corollary.

\begin{cor}
\label{snyst}
$\lim_{s\rightarrow \pm \infty}\theta = +\infty$. \\In other words
  $\int_{s_0}^\infty k\hspace{0.1cm} ds = +\infty$ and   $\int_{-\infty}^{s_0} k\hspace{0.1cm} ds = -\infty$, so each arm has infinite total curvature.
\end{cor}

This result concludes the proof of the theorem. Note that in the purely rotational case ($B=0$) simple calculations yield that the function $r^2-\frac{2}{A}\theta$ is constant, which is an interesting fact. This also gives another proof of the fact that the curve can't cross itself.

\section{Curves which rotate and shrink}

For each $A\neq0$ and $B< 0$, we have the following description of the curves satisfying Equation \eqref{jafnan}. We omit the circle of radius $\frac{1}{\sqrt{-B}}$, where the rotation is vacuous.

\begin{thm}
\label{snuainn}
In this case there are two types of curves.

1) Curves such that the limiting behaviour when going along the curve in each direction is wrapping around the circle of radius  $\frac{1}{\sqrt{-B}}$, clockwise if $A<0$ and counterclockwise if $A>0$.  These curves form a one-dimensional family.

2) Curves such that the curvature never changes sign and the two ends behave very differently. One end wraps around the circle with radius $\frac{1}{\sqrt{-B}}$ in its limiting behaviour, clockwise if $A<0$ and counterclockwise if $A>0$. The other end spirals infinitely many circles around the origin out to infinity and has infinite total curvature. The curvature goes to $0$ along it, and its limiting growing direction is $-B-iA$ times the location. There is at least one curve of this type, and we call it the comet spiral.

These curves rotate and shrink with angular function $\frac{A}{2B}\log(2Bt+1)$ and scaling function $\sqrt{2Bt+1}$ under the CSF. A singularity forms at time $t=-\frac{1}{2B}$. The curves of type $1$ are bounded, so they disappear into the origin. 
\end{thm}

Curves of this kind can be seen in Figures \ref{A1Bm005samhverfi}-\ref{A1Bm5halastjarnan}.\\

The proof will be given through a series of lemmas. By reflecting the curve $X$ if necessary, we may assume $A>0$. By rescaling, we may assume $A = 1$ and $B = -\beta^2$ for some $\beta > 0$. Then we have the system
\begin{equation*}\left\{\begin{aligned}
x' &= xy+1,\\
y' &= -x^2+\beta^2.
\end{aligned}\right.\end{equation*}
By symmetry, it will suffice to look at the behaviour of the solutions in the right half-plane. There the system has a fixed point $(\beta,-\frac{1}{\beta})$ which corresponds to $X$ being the circle of radius $\frac{1}{\beta}$ parametrized counterclockwise (making the rotation vacuous). The eigenvalues of the linearization of the system around the fixed point are
\begin{equation*}
\frac{-1\pm \sqrt{1-8\beta^4}}{2\beta}.
\end{equation*}
Since the real part is always negative, the fixed point is a sink, i.e., asymptotically stable. Moreover, it is a spiral if $\beta > 8^{-1/4}$ and a node if $0 < \beta < 8^{-1/4}$. It will turn out to be asymptotically stable in the whole right half-plane. We begin with a simple lemma.

\begin{lem}
\label{klart}
The curvature $x$ has at most one zero, is negative before it and positive after it.  Moreover, if $x'(s)=0$, then $x$ has a maximum at $s$ if $x(s)>\beta$ and a minimum if $0<x(s)<\beta$.
\end{lem}
\begin{proof}
The first statement is clear since $x'(s) = 1$ if $x(s)=0$. The second statement follows from the fact that $x''(s) = -x(s)(x^2(s)-\beta^2)$ if $x'(s)=0$.
\end{proof}

\begin{lem}
\label{spiralinn}
If $x_0 = 0$, then $\lim_{s\rightarrow \infty}(x,y) = (\beta,-\frac{1}{\beta})$.
\begin{proof}
First assume $x<\beta$ always holds. Then $y'>0$, so $y$ is increasing. If $y$ attains the value $0$, then $x' \geq 1$ from that point onwards, which implies $x$ will eventually reach the value $\beta$, a contradiction. Hence, $y$ is bounded, so it has a finite limit. It is clear that $y''$ is bounded since both $x$ and $y$ are bounded, and hence, by Barbalat's lemma, $y' \rightarrow 0$. But then $x \rightarrow \beta$, and since both $x$ and $y$ have limits, the limit has to be the fixed point $(\beta,-\frac{1}{\beta})$.

In the other case, $x$ reaches $\beta$. Let $s_1$ be the first time that happens. Then $y'(s_1)=0$, and we must have $y(s_1) >-\frac 1 \beta$ since $x'(s_1) > 0$. Again we have two possibilities. First assume $x>\beta$ holds for all $s > s_1$. Then $y'<0$, so $y$ is decreasing. If $y$ reaches $-\frac 1 \beta$, then we will have $x' \leq -\delta < 0$, so $x$ goes down to $\beta$, a contradiction. Therefore $y$ is bounded from below and has a finite limit. Also note that $x$ can't always be increasing, since then $y'$ would be decreasing and $y$ would go to $-\infty$. Therefore, $x$ reaches a maximum and after that will be decreasing, by Lemma \ref{klart}. Hence, $x$ also has a finite limit, and thus $(x,y)$ has the fixed point as its limit.
 
The other possibility is that $x$ reaches $\beta$ again. Let $s_2$ be the first time this happens. Then $y'(s_2)=0$, and we must have $y(s_2) <-\frac 1 \beta$ since $x'(s_2) < 0$.  By repeating the argument from the first paragraph we get that either $(x,y)$ goes directly to the fixed point or there exists an $s_3$ such that $x(s_3)=\beta$ and $y(s_3) > -\frac 1 \beta$. Also note that since $x_0=0$, $x>0$ for all $s>0$ and the $(x,y)$ trajectory can't cross itself, we must have $y(s_3) <  y(s_1)$, i.e., the curve spirals inwards.

Now, let $K$ be the compact region in the plane with boundary consisting of the $(x,y)$ trajectory between $s_1$ and $s_3$ and the line segment between those two endpoints. Then it is clear that the trajectory after $s_3$ will never leave $K$, since it can neither cross its own path nor the line segment because there the arrows of the $(F,G)$ vector field are all pointing into $K$. Hence the result of the lemma follows from the Poincar\'e-Bendixson theorem (see \cite{gh}) and the following lemma. 
\end{proof}
\end{lem}

\begin{lem}
\label{dulli}
There are no periodic orbits in the right half-plane.
\end{lem}
\begin{proof}
Define the scalar field $H(x,y) = \frac 1 x$ which is smooth in the right half-plane. Since
\begin{equation*}
\frac{\partial}{\partial x} (HF) + \frac{\partial}{\partial y} (HG) = -\frac{1}{x^2}
\end{equation*}
does not change sign in the right half-plane, by Dulac's criterion (see \cite{gh}) there can be no periodic orbits there.
\end{proof}

\begin{lem}
\label{spiralhinir}
The fixed point is asymptotically stable in the whole right half-plane. In other words, if $x_0 >0$, then $\lim_{s\rightarrow \infty}(x,y) = (\beta,-\frac{1}{\beta})$.
\end{lem}
\begin{proof}
In Lemma \ref{spiralinn}, we proved this result for trajectories in the right half-plane which cross the $y$-axis. However, the only place we used that extra condition was to show that the trajectory can't spiral outwards. But if one trajectory spirals inwards, then all the others also have to spiral inwards since otherwise they would be forced to cross the one spiraling inwards. Hence the same argument as before works for all trajectories in the right half-plane.
\end{proof}

But what do the trajectories in the right half-plane do when we go backwards in time? Let $\Omega$ be the region in the plane defined by $\Omega=\{(x,y): 0<x<\beta\}$. By the following lemma, it is clear that if a trajectory leaves $\Omega$ through the right boundary, then it has to come back in. 

\begin{lem}
\label{afturarassabak}
If $x(s_0) > \beta$, there exists $s < s_0$ such that $x(s)=\beta$. 
\end{lem}
\begin{proof}
Assume not. Then $x > \beta$ for all $s<s_0$, so $y'<0$, and thus $y$ is decreasing (increasing in backwards direction). If $y$ manages to reach 0 when going backwards, then $x'\geq 1$, so $x$ will also reach $\beta$, a contradiction. Therefore, $y$ is bounded from above and has a finite limit at $-\infty$. If $x'<0$ for all $s < s_0$, then $x$ is decreasing, so $y'$ is decreasing and hence $y$ goes to $+\infty$ when $s\rightarrow -\infty$, a contradiction. Hence, $x$ reaches a maximum and before that is increasing, by Lemma \ref{klart}. Thus $x$ also has a finite limit when $s\rightarrow -\infty$. However, the limit of $(x,y)$ when $s\rightarrow -\infty$ would have to be a fixed point, and the only one in the right half-plane is $(\beta,-\frac 1 \beta)$. This can never happen since the fixed point repels trajectories when we go backwards in time.
\end{proof} 

Now, fix any trajectory and let $z_0$ be any point on the right boundary of $\Omega$ where the trajectory leaves $\Omega$ when we go backwards in time. Let $\gamma$ be the open half-line that forms the part of the right boundary strictly below $z_0$.  Take any trajectory that crosses the $y$-axis. It intersects $\gamma$ only finitely many times, say $k$ times, because it goes to the fixed point in the forward direction. But that means our trajectory can intersect $\gamma$ at most $k$ times, since if we look at the sequence of intersection points of the two trajectories with $\gamma$, no two consecutive points can belong to the same trajectory without the two trajectories intersecting. Hence, our trajectory will ultimately stop leaving $\Omega$ through the right boundary.

 After that there are two possible behaviours in the backwards direction:\\\\
1) The trajectory leaves $\Omega$ through the left boundary, i.e., crosses the $y$-axis. By Lemma \ref{spiralinn} and symmetry, it is then clear that $\lim_{s\rightarrow -\infty}(x,y) = (-\beta,\frac{1}{\beta})$, i.e., the trajectory goes to the fixed point in the left half-plane which is a source. Since $r^2 (1 + \beta^4) = (x^2+y^2)$, we have
\begin{equation*}
\lim_{s\rightarrow \pm\infty}r = \frac 1 \beta.
\end{equation*}
Hence, each arm of the curve $X$ has as its limiting behaviour wrapping counterclockwise around the circle of radius $\frac 1 \beta$. These curves form a one-dimensional family that can for example be parametrized by the value $y_0$ when $x_0=0$. \\\\
2) The trajectory stays in $\Omega$ forever. Then by the following lemma we have
\begin{equation*}
\lim_{s\rightarrow -\infty}x = 0, \lim_{s\rightarrow -\infty}y = -\infty, \hspace{0.2cm} \text{ and hence}\hspace{0.2cm} \lim_{s\rightarrow -\infty}r = \infty.
\end{equation*}
Hence, the curve $X$ has one arm wrapping counterclockwise around the circle of radius $\frac 1 \beta$ and the other one going out to infinity. The same arguments as in Lemmas \ref{markgildi}, \ref{philimit} and Corollaries \ref{stefna} and  \ref{snyst} show that the arm has infinite total curvature, spirals infinitely many circles around the origin and its limiting growing direction is $-B-iA$ times the location. The curvature of $X$ never changes sign.

\begin{lem}
\label{serstaki}
The trajectories of type $2$ satisfy
\begin{equation*}
\lim_{s\rightarrow -\infty}x = 0 \text{ and } \lim_{s\rightarrow -\infty}y = -\infty.
\end{equation*}
\end{lem}
\begin{proof}
By Lemma \ref{klart}, $x$ has at most one extremum (after entering $\Omega$ for the last time), and since it is bounded, it has a finite limit as $s \rightarrow -\infty$.
Since the curve remains in the region $\Omega$, we have $y'>0$, so when going backwards in time $y$ is decreasing. If $y$ were bounded from below it would have a finite limit as $s\rightarrow -\infty$. But then $(x,y)$ would have to go to the fixed point, a contradiction since it repels trajectories when we go backwards in time. Hence, it is clear that $y\rightarrow -\infty$ as $s\rightarrow -\infty$.

Now, since $x'=xy+1$, it is clear that if $\lim_{s\rightarrow -\infty}x  >0$, then $\lim_{s\rightarrow -\infty}x' = -\infty$, a contradiction. Hence, $\lim_{s\rightarrow -\infty}x  =0$.
\end{proof}

So the remaining question is: Does there exist a curve of type 2? 
\begin{lem}
\label{2ertil}
There exists a curve of type $2$.
\end{lem}
\begin{proof}
Since the eigenvalues of the linearization of the vector field $(F,G)$ around the fixed point $(\beta,-\frac{1}{\beta})$ have negative real part, there exists a small circle $C$ around the fixed point such that everywhere on the circle, $(F,G)$ is pointing into it. That means every trajectory in the right half-plane crosses $C$ in a unique point. Now, map each point $(0,y_0)$ to the unique point on $C$ where the trajectory through $(0,y_0)$ intersects it. Then the $y$-axis maps bijectively and smoothly onto some relatively open, connected subset $U$ of $C$, i.e., an open interval which thus cannot be all of $C$. Let $z_0$ be any point in the complement of $U$ in $C$. Then the backwards trajectory from $z_0$ cannot cross the $y$-axis, so it must be of type 2. 
\end{proof}

This finishes the proof of the theorem. Note that these curves are usually not embedded but have loops. By Equation \eqref{flatarmal} and the fact that the motion is self-similar, we get that the area inside each loop decreases with the fixed rate $-(\pi+\alpha)$, where $\alpha \in (0,2\pi)$ is the interior angle at the double point. The loop shrinks to a point at time $t=-\frac{1}{2B}$, and therefore we have the following relationship between the interior angle and the area inside the loop:
\begin{equation*}
\text{Area inside loop} = -\frac{\pi+\alpha}{2B}.
\end{equation*}
Of course, we only get $\alpha = \pi$ (smooth loop) when the curve $X$ is the circle of radius $\frac{1}{\sqrt{-B}}$. Since $\alpha > 0$ it follows that each loop has at least half the area of the circle, which is an interesting fact.

\section{Shrinking curves}

The remaining two cases are already more or less known (see e.g.~ \cite{al}, \cite{ish} and \cite{urb}),  but we cover them here for the sake of completeness.

For $A=0$ and each $B<0$, we have the following description of the curves satisfying Equation \eqref{jafnan}. We omit the straight lines through the origin where the scaling is vacuous.

\begin{thm}
\label{inn} Each of the curves is contained in an annulus around the origin and consists of a series of identical excursions between the two boundaries of the annulus. The curvature is an increasing function of the radius and never changes sign. The inner and outer radii of the annulus, $r_{min}$ and $r_{max}$, satisfy $r_{min}e^{\frac{Br_{min}^2}{2}} = r_{max}e^{\frac{Br_{max}^2}{2}}$
 and take on every value in $(0,\tfrac{1}{\sqrt{-B}}]$ and $[\tfrac{1}{\sqrt{-B}},\infty)$, respectively.

The curves form a one-dimensional family parametrized by $r_{min}$ and are divided into two sets:

1) Closed curves, i.e., immersed $\mathbf S^1$ (Abresch-Langer curves). In addition to the circle, there is a curve with rotation number $p$ which touches each boundary of the annulus $q$ times for each pair of mutually prime positive integers $p,q$ such that $\frac{1}{2} < \frac p q < \frac{\sqrt{2}}{2}$.

2) Curves whose image is dense in the annulus.

Under the CSF these curves shrink with scaling function $g(t) = \sqrt{2Bt+1}$ until they disappear into the origin at time $t=-\frac{1}{2B}$. 
\end{thm}

Curves of this kind can be seen in Figures \ref{AL1}-\ref{AL4}.\\

The proof is contained in a series of lemmas. 
By rescaling if necessary, we may assume $B=-1$. Then we have the following system for $x$ and $y$:
\begin{equation*}\left\{\begin{aligned}
x' &= xy,\\
y' &= -x^2+1.
\end{aligned}\right.\end{equation*}
From the first equation we immediately conclude
\begin{equation}
\label{xafy}
x(s) = x_0e^{\int_0^sy(s')ds'},
\end{equation}
so the curvature $x$ does not change sign. When $x_0 = 0$ we just get the solution $(x,y)$ = $(0,y_0+s)$ which corresponds to $X$ being a straight line, making the scaling vacuous. By parametrizing the curves backwards if necessary, we may assume $x_0>0$ and hence $x>0$. Note that the system has the fixed point $(1,0)$ corresponding to $X$ being the unit circle parametrized counterclockwise.

We start with a technical lemma.
\begin{lem}
\label{ynidur}
If $y_0 > 0$, then there exist $s_1 < 0 < s_2$ such that $y(s_1)=0$ and $y(s_2)=0$.
\end{lem}
\begin{proof}
Assume $y>0$ for all $s\geq0$. Then $x'>0$, so $x$ is increasing. If $x<1$ holds for all $s\geq 0$, then $y' = -x^2+1 > 0$, so $y$ is increasing. But then by Equation \eqref{xafy}, we have $x \rightarrow \infty$ when $s\rightarrow \infty$, contradicting $x<1$. Hence, $x$ must reach values above $1$ which implies $y' \leq -\delta < 0$, so $y$ has to go down to $0$, a contradiction. The other case is treated similarly. 
\end{proof}
\begin{lem}
\label{lota}
The solution $(x,y)$ is periodic.
\end{lem}
\begin{proof}
Since $F(x,-y)=-F(x,y)$ and $G(x,-y)=G(x,y)$, it is clear that $s \mapsto (x(-s),-y(-s))$ is also a solution to the system of ODEs. Hence, if $y_0=0$, the initial values are the same, so by uniqueness, we get that $(x(-s),-y(-s))=(x(s),y(s))$. But this means that the trajectory is symmetric with respect to reflection across the $x$-axis. By Lemma \ref{ynidur} each trajectory crosses the $x$-axis twice and therefore has to be closed, i.e., periodic.
\end{proof}
\begin{cor}
\label{rlota}
Both the curvature $x$ and the radius $r$ are periodic.
\end{cor}
The phase portrait of our dynamical system therefore consists of the $y$-axis, the fixed point $(1,0)$ and closed trajectories around $(1,0)$. Note that for each of these the curvature $x$ takes its maximum and minimum on the $x$-axis, since $x'=xy$.

Direct calculations give us the following lemma.
\begin{lem}
\label{fastarass}
The function $x\hspace{0.05cm}e^{-\frac{x^2+y^2}{2}}$ is constant.
\end{lem}
Remember that $(A,B) = (0,-1)$, so we have $r^2 = x^2 + y^2$. Now, imagine we parametrize our solution such that $y_0=0$. By Lemma \ref{fastarass}, we have the following relationship between the curvature $x$ and the radius $r$:
\begin{equation*}
x = x_0e^{\frac{r^2-x_0^2}{2}},
\end{equation*}
where $x_0$ can be either the minimum of the curvature, $x_{min}$ or its maximum, $x_{max}$. In particular, $r$ is an increasing function of $x$, so we have the following corollary.
\begin{cor}
\label{rmin}
The radius takes its maximum and minimum at the same time as the curvature $x$. Moreover, $r_{min} = x_{min}$, $r_{max} = x_{max}$ and $r_{min}e^{-\frac{r_{min}^2}{2}} = r_{max}e^{-\frac{r_{max}^2}{2}}$.
\end{cor}
So the curve $X$ consists of infinitely many identical excursions between the two boundaries of the annulus with inner radius $r_{min}$ and outer radius $r_{max}$. The curvature is an increasing function of the radius and reaches its lowest and highest value at the innermost and outermost points, respectively.

Now let's examine a whole excursion, i.e., we go from the outer boundary, touch the inner boundary and go back out the outer boundary. This corresponds to the $(x,y)$ trajectory going one period. Therefore, the difference between the values of the location angle $\phi$ is the same as the difference between the values of the tangent angle $\theta$. This angle difference, $\Delta \theta$, is given by the integral of the curvature $x$ over one period of the $(x,y)$ trajectory. Now we have two possibilities:\\
\\
1) The angle difference $\Delta \theta$ is of the form $\frac{p2\pi}{q}$ for some positive integers $p,q$. Then after $q$ excursions the curve $X$ closes up. Thus, $X$ is a closed curve, an immersed $\mathbf S^1.$ The number $p$ is the rotation number of the curve, and $q$ is the number of times it touches each boundary of the annulus.\\
\\
2) The angle difference  $\Delta \theta$ is not of this form. In this case the curve never closes up and takes infinitely many excursions. It is dense in the annulus.\\

In \cite{al}, Abresch and Langer investigated $\Delta \theta$ as a function of $r_{max}$ and showed that it is a continuous decreasing function with limits
\begin{equation*}
\lim_{r_{max} \rightarrow 1+}\Delta\theta = \sqrt 2 \pi \hspace{0.2cm} \text{ and }\hspace{0.2cm} \lim_{r_{max} \rightarrow \infty}\Delta\theta = \pi.
\end{equation*}
Therefore, we get a closed curve for all mutually prime positive integers $p,q$ such that $\frac{1}{2} < \frac{p}{q} < \frac{\sqrt{2}}{2}$. 
This finishes the proof of the theorem. 

\section{Expanding curves}
\label{ultimo}

For $A=0$ and each $B>0$, we have the following description of the curves satisfying Equation \eqref{jafnan}. 

\begin{thm}
\label{ut}
Each of the curves is convex, properly embedded and asymptotic to the boundary of a cone with vertex at the origin. It is the graph of an even function.

The curves form a one-dimensional family parametrized by their distance to the origin, which can take on any value in $[0,\infty)$.

Under the CSF these curves expand forever as governed by the scaling function $g(t) = \sqrt{2Bt+1}$.
\end{thm}

Curves of this kind can be seen in Figures \ref{graf1} and \ref{graf2}.\\

As before, the proof will be given through a series of lemmas. By rescaling we may assume $B=1$. Then $(x,y)$ is a solution to the system
\begin{equation*}\left\{\begin{aligned}
x' &= xy,\\
y' &= -x^2-1.
\end{aligned}\right.\end{equation*}
From the first equation we immediately get
\begin{equation}
\label{xafy2}
x(s) = x_0e^{\int_0^sy(s')ds'},
\end{equation}
so the curvature $x$ does not change sign. When $x_0 = 0$ we just get the solution $(x,y)$ = $(0,y_0-s)$ which corresponds to $X$ being a straight line (making the scaling outwards vacuous). By parametrizing the curve $X$ backwards if necessary, we may assume $x_0>0$ and hence $x>0$.

Since $y' \leq -1$, it is clear that $\lim_{s\rightarrow \pm \infty}y = \mp \infty$, so by \eqref{xafy2} we get $\lim_{s\rightarrow \pm \infty} x = 0$. Also, since $F(x,-y)=-F(x,y)$ and $G(x,-y)=G(x,y)$, we get the following as in the proof of Lemma \ref{lota}, if we parametrize the solution such that $y_0=0$.
\begin{lem}
\label{typpi}
The curvature $x$ is an even function and $y$ is an odd function.
\end{lem}
Thus, the angle $\theta$ is an odd function (remember we assume $\theta_0=0$), and since $X$ is given by \eqref{EnneinXjafnan}, we have the following corollary.
\begin{cor}
\label{saurmundur}
The curve $X$ is symmetric with respect to reflection across the $y$-axis.
\end{cor}
Direct calculations yield this lemma.
\begin{lem}
\label{fastarassiprump}
The function $x\hspace{0.05cm}e^{\frac{x^2+y^2}{2}}$ is constant.
\end{lem}
Since $(A,B)=(0,1)$ we have $r^2 = x^2+y^2$. Thus, the following relationship between the curvature $x$ and the radius $r$ holds:
\begin{equation*}
x = x_0e^{\frac{x_0^2-r^2}{2}},
\end{equation*}
where $x_0 = x_{max}$, the maximum value of the curvature, since $x'=xy$.
In particular, $r$ is a decreasing function of $x$, so the following is true.
\begin{cor}
The radius $r$ takes its minimum when the curvature $x$ takes its maximum. Moreover, $r_{min}=x_{max}$.
\end{cor}

\begin{lem}
\label{totaljota}
The total curvature of $X$ is strictly less than $\pi$.
\end{lem}
\begin{proof}
The total curvature is given by the integral
\begin{equation*}\begin{aligned}
\int_{-\infty}^{\infty}x(s)ds &= 2\int_0^\infty x(s)ds = 2 \int_{r_{min}}^\infty\frac{x(r)r}{\sqrt{r^2-x^2(r)}}dr\\
&< 2 \int_{r_{min}}^\infty\frac{x(r)r}{\sqrt{r_{min}^2-x^2(r)}}dr\\
&= 2 \int_{r_{min}}^\infty\frac{r}{\sqrt{e^{r^2-r_{min}^2}-1}}dr\\
&= 2\int_0^\infty\frac{1}{u^2+1}du\\
&= \pi,
\end{aligned}\end{equation*}
where we put $u=\sqrt{e^{r^2-r_{min}^2}-1}$ in the second to last equality.
\end{proof}

This result and Corollary \ref{saurmundur} yield that $X$ is the graph of an even function and hence embedded. It also follows that $X$ is convex.  In \cite{ish}, Ishimura shows that the curve is asymptotic to a cone with vertex at the origin, finishing the proof of our theorem. Moreover, he shows that the total curvature is an increasing continuous function of the distance to the origin which gives a homeomorphism between $[0,\infty)$ and $[0,\pi)$.

\section{Figures}
\label{myndir}

\begin{figure}[h!]
\centering 
\includegraphics[totalheight=3.4in]{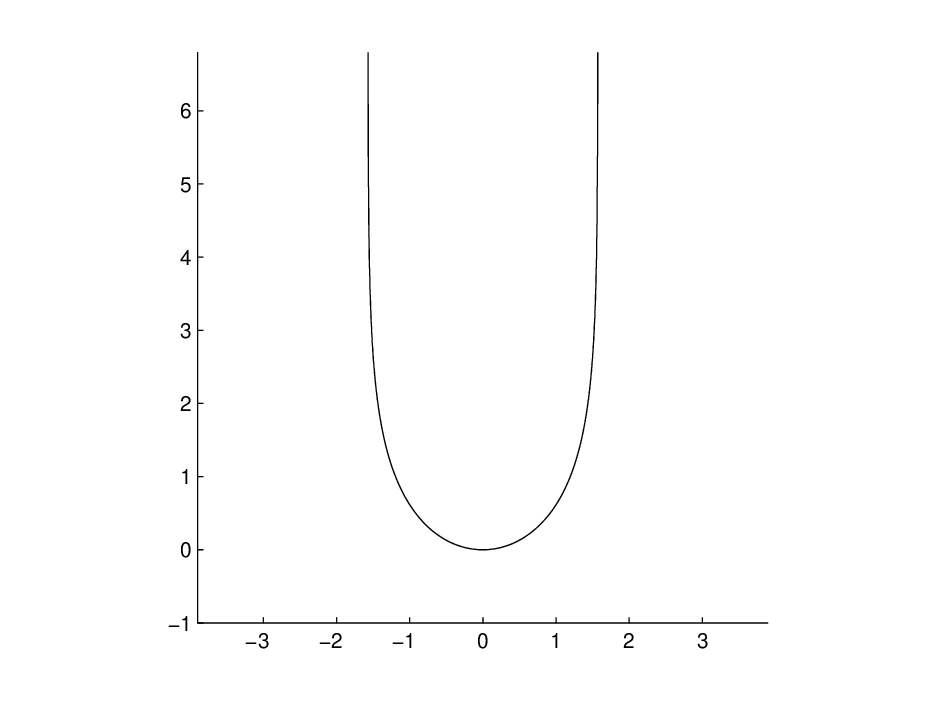} 
\caption{The Grim Reaper curve, which translates under the flow.} 
\label{GrimReaper} 
\end{figure}

\clearpage

\begin{figure}
\centering 
\includegraphics[totalheight=3.4in]{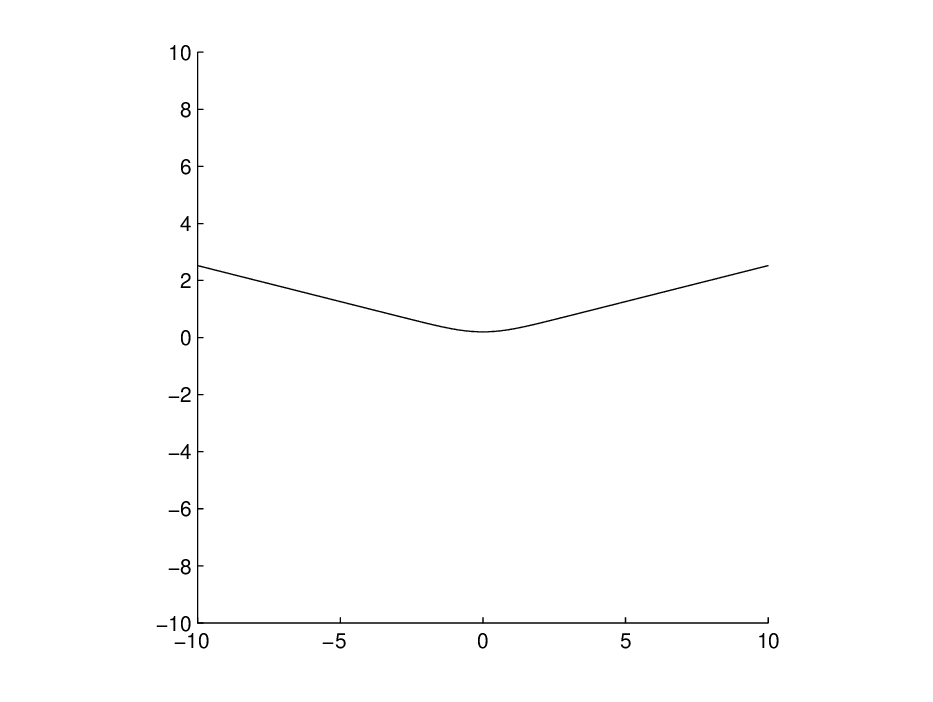} 
\caption{$A = 0$, $B = 1$ (expands) The curve is asymptotic to the boundary of a cone with vertex at the origin.} 
\label{graf1} 
\end{figure}

\begin{figure}
\centering 
\includegraphics[totalheight=3.4in]{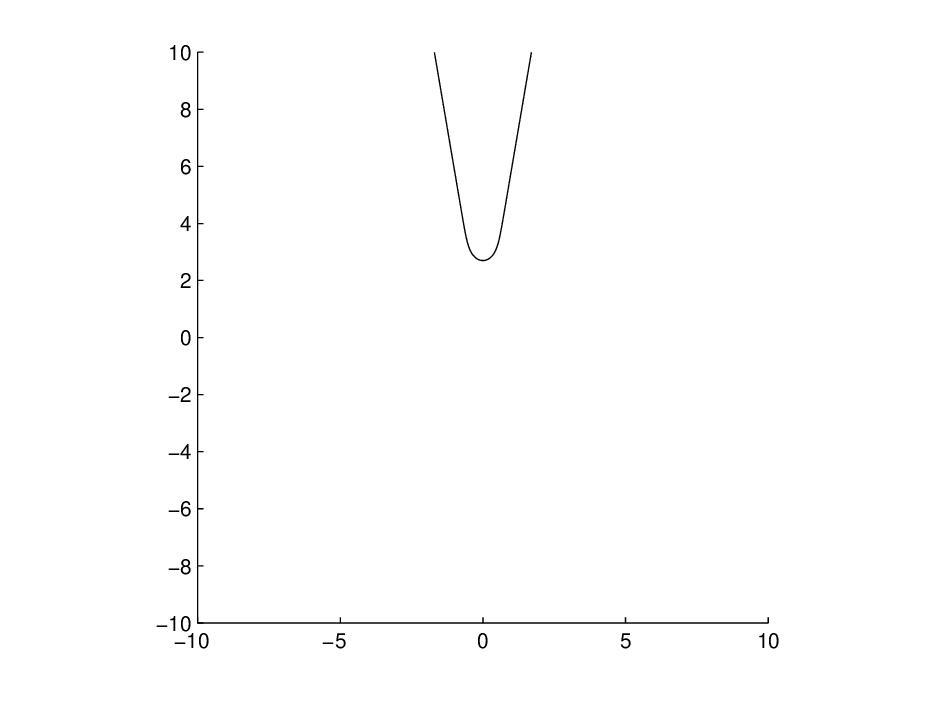} 
\caption{$A = 0$, $B = 1$ (expands) The total curvature increases with the distance to the origin.} 
\label{graf2} 
\end{figure}

\begin{figure}
\centering 
\includegraphics[totalheight=3.4in]{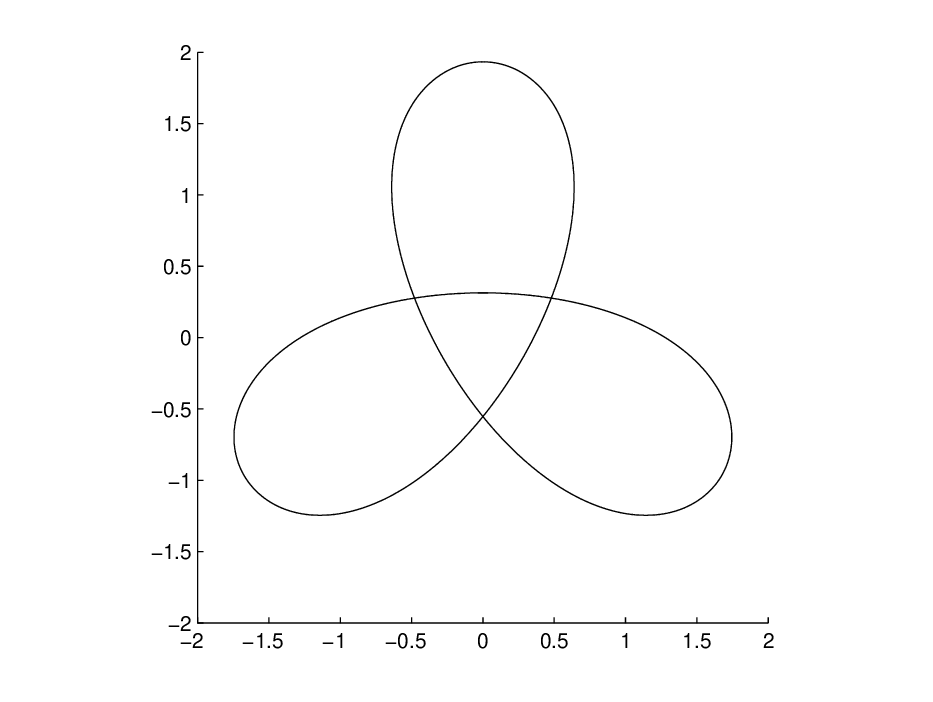} 
\caption{$A = 0$, $B = -1$ (shrinks) Abresch-Langer curve with $p=2$ and $q=3$.} 
\label{AL1} 
\end{figure}

\begin{figure}
\centering 
\includegraphics[totalheight=3.4in]{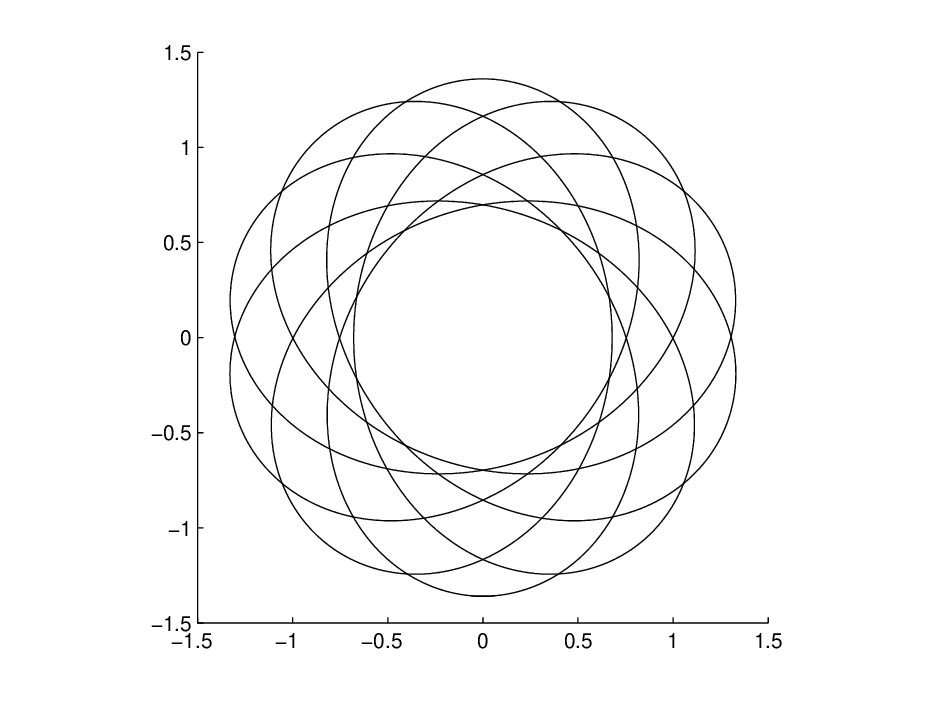} 
\caption{$A = 0$, $B = -1$ (shrinks) Abresch-Langer curve with $p=7$ and $q=10$.} 
\label{AL2} 
\end{figure}

\begin{figure}
\centering 
\includegraphics[totalheight=3.4in]{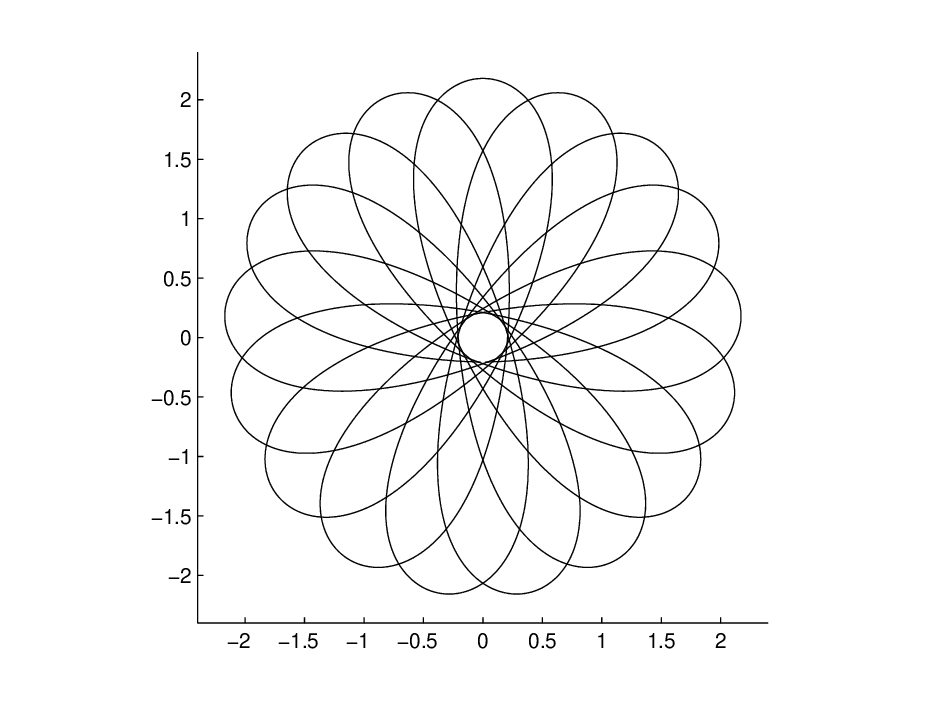} 
\caption{$A = 0$, $B = -1$ (shrinks) Abresch-Langer curve with $p=20$ and $q=31$. } 
\label{AL3} 
\end{figure}

\begin{figure}
\centering 
\includegraphics[totalheight=3.4in]{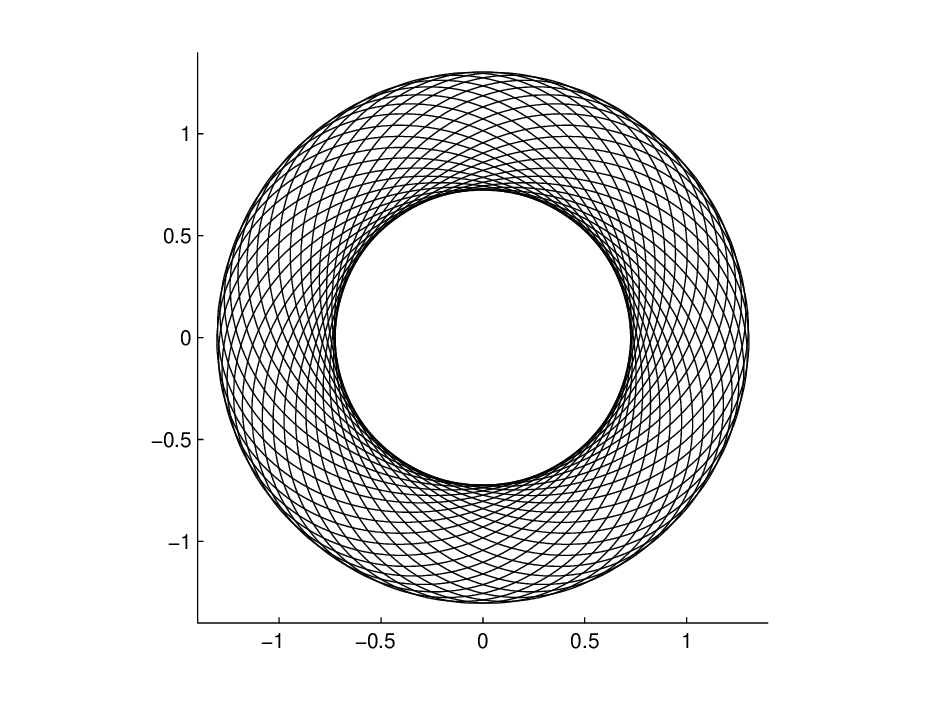} 
\caption{$A = 0$, $B = -1$ (shrinks) Abresch-Langer curve where the values of $p$ and $q$ are very high. } 
\label{AL4} 
\end{figure}

\begin{figure}
\centering 
\includegraphics[totalheight=3.4in]{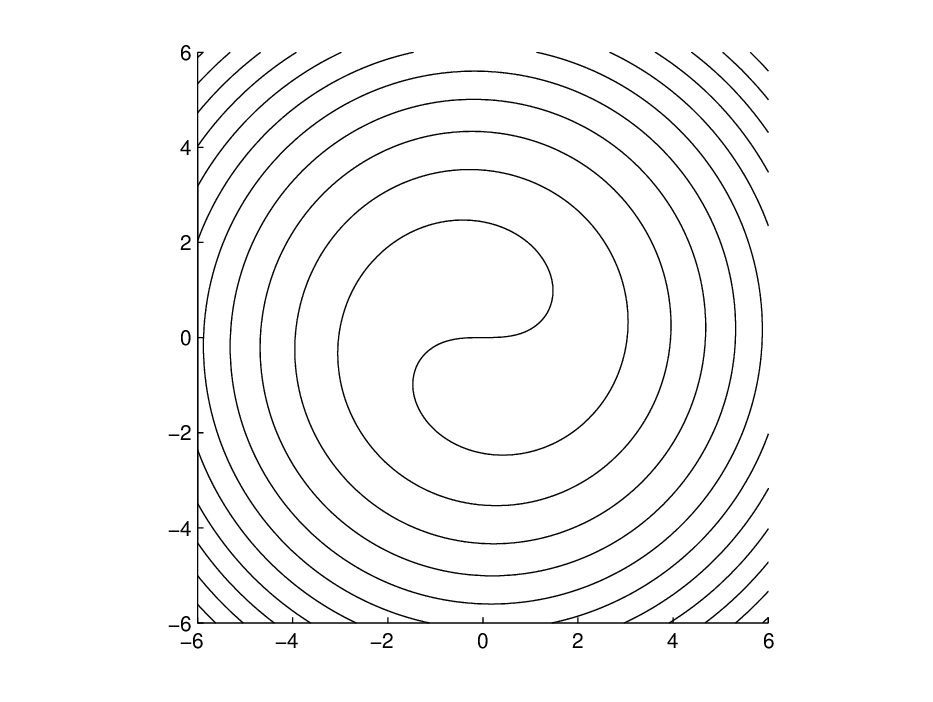} 
\caption{$A = 1$, $B = 0$ (rotates) The symmetric curve is Altschuler's yin-yang spiral.} 
\label{A1B0samhverfi} 
\end{figure}

\begin{figure}
\centering 
\includegraphics[totalheight=3.4in]{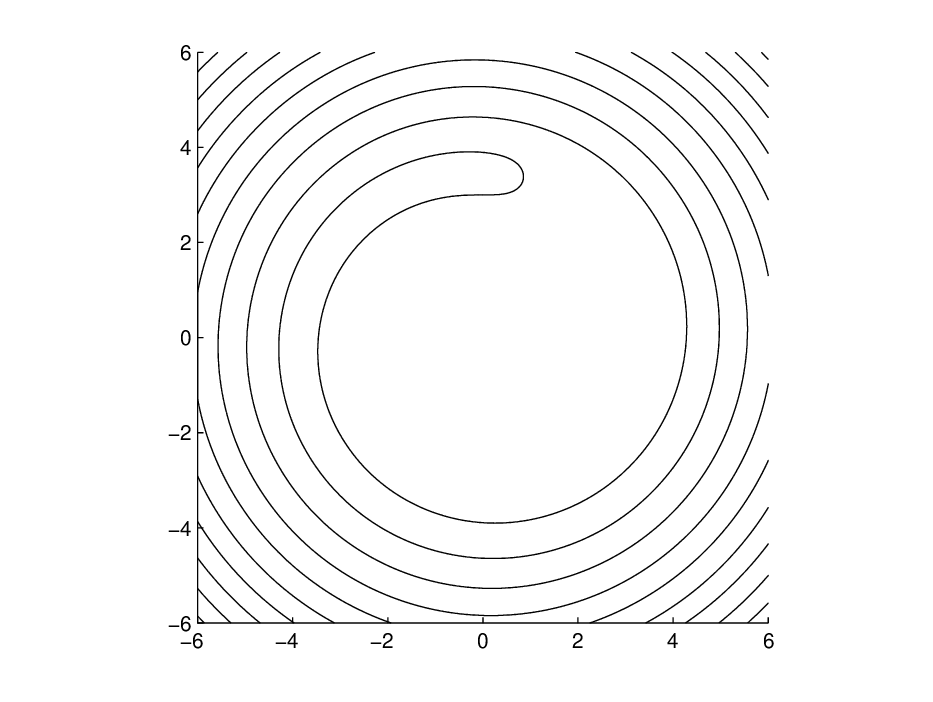} 
\caption{$A = 1$, $B = 0$ (rotates) As the distance to the origin increases, the tip of the spiral looks more and more like the Grim Reaper.} 
\label{A1B0b3} 
\end{figure} 

\begin{figure}
\centering 
\includegraphics[totalheight=3.4in]{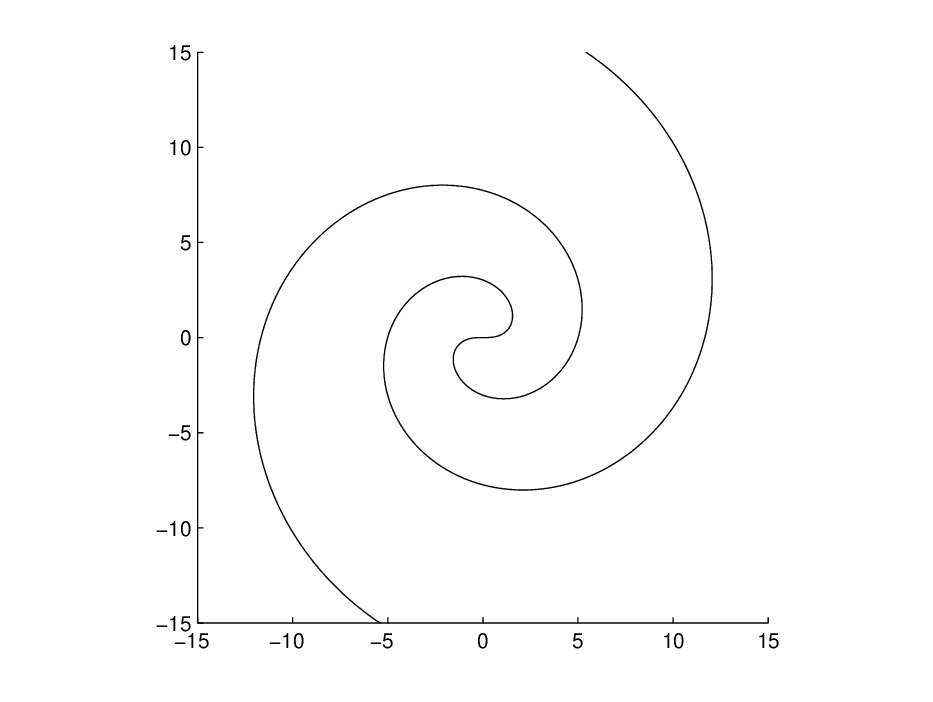} 
\caption{$A = 1$, $B = 0.25$ (rotates and expands) The symmetric curve looks like a less dense version of the the yin-yang spiral.} 
\label{A1B025samhverfi} 
\end{figure}

\begin{figure}
\centering 
\includegraphics[totalheight=3.4in]{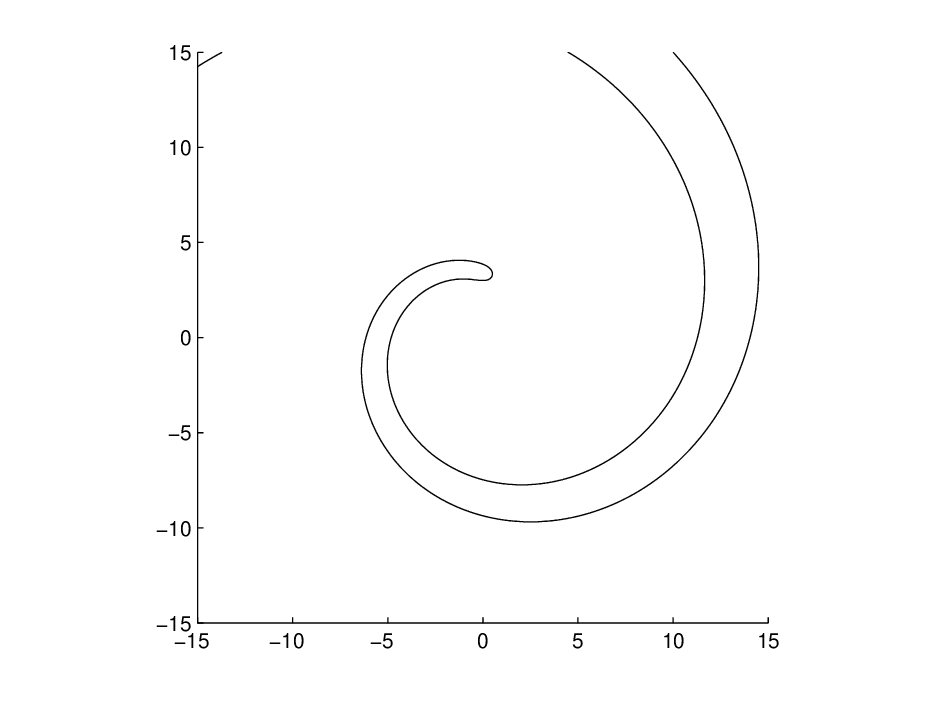} 
\caption{$A = 1$, $B = 0.25$ (rotates and expands) It gets thinner as the distance to the origin increases.} 
\label{A1B025b3} 
\end{figure}

\begin{figure}
\centering 
\includegraphics[totalheight=3.4in]{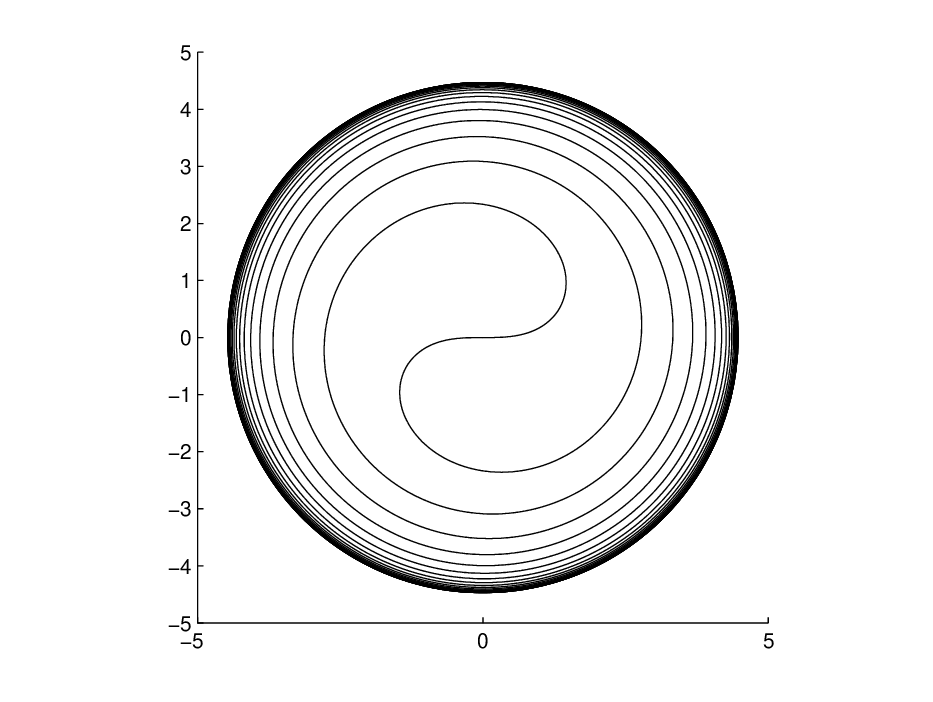} 
\caption{$A = 1$, $B = -0.05$ (rotates and shrinks). The symmetric curve is similar to the yin-yang spiral but only spirals out to the limiting circle. It looks embedded for values of $B$ close to 0.} 
\label{A1Bm005samhverfi} 
\end{figure}

\begin{figure}
\centering 
\includegraphics[totalheight=3.4in]{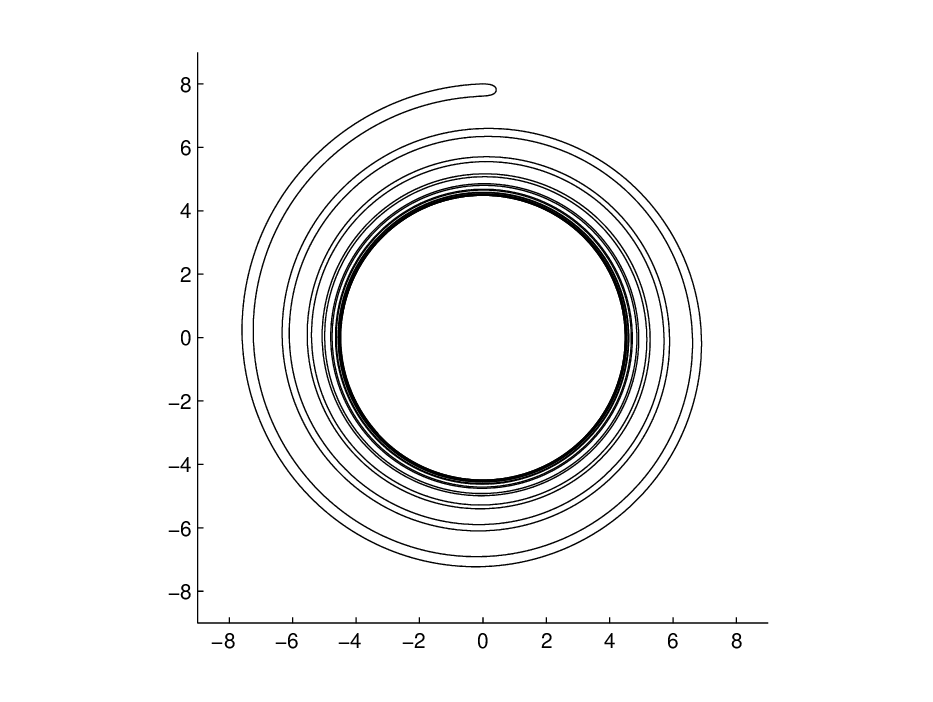} 
\caption{$A = 1$, $B = -0.05$ (rotates and shrinks) Here we start on the outside of the circle and spiral inwards towards it.} 
\label{A1Bm005b8} 
\end{figure}

\begin{figure}
\centering 
\includegraphics[totalheight=3.4in]{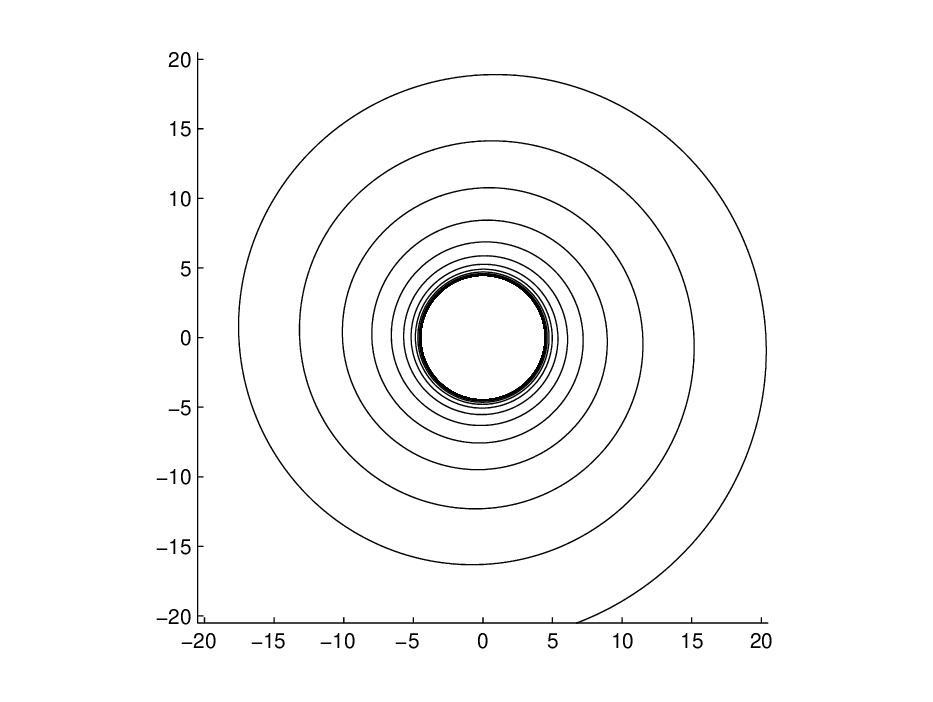} 
\caption{$A = 1$, $B = -0.05$ (rotates and shrinks) This is the comet spiral. It seems to be embedded for values of $B$ close to 0.} 
\label{A1Bm005halastjarnan} 
\end{figure}

\begin{figure}
\centering 
\includegraphics[totalheight=3.4in]{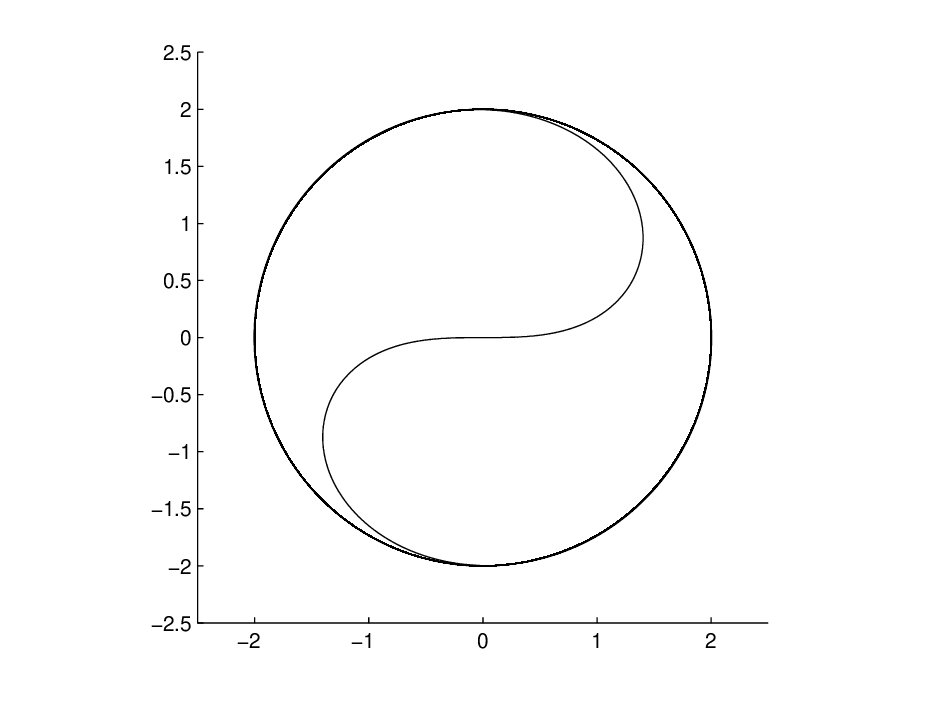} 
\caption{$A = 1$, $B = -0.25$ (rotates and shrinks) Here the symmetric curve is similar to the yin-yang symbol. Is it still embedded?} 
\label{A1Bm025samhverfi} 
\end{figure}

\begin{figure}
\centering 
\includegraphics[totalheight=3.4in]{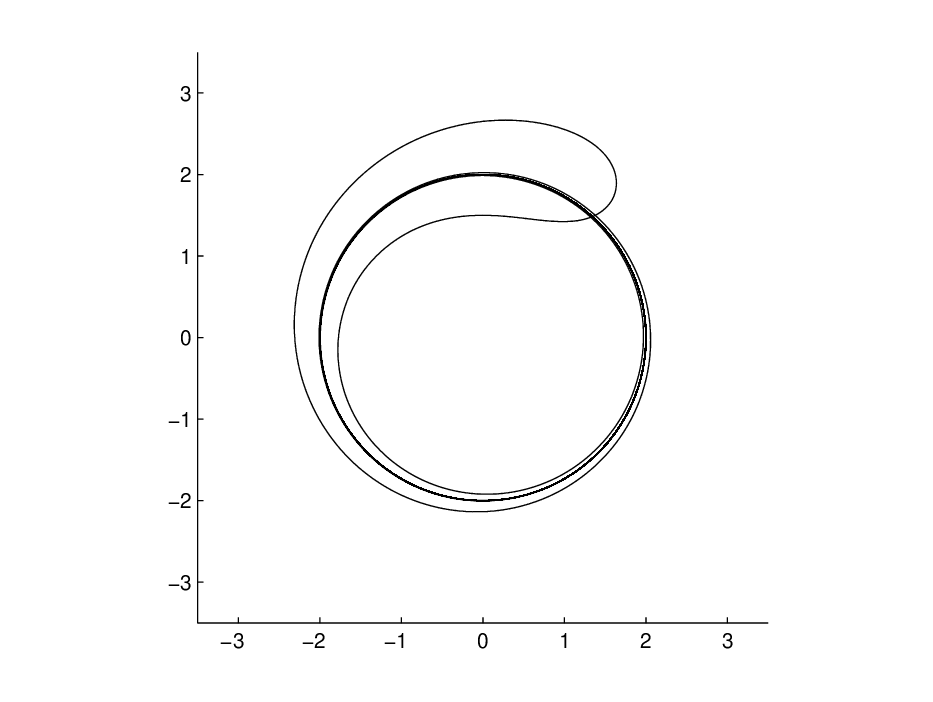} 
\caption{$A = 1$, $B = -0.25$ (rotates and shrinks) A typical curve of this kind looks like a spirit coming out of a circle. Each loop has at least half the area of the circle.} 
\label{A1Bm025b15} 
\end{figure}

\begin{figure}
\centering 
\includegraphics[totalheight=3.4in]{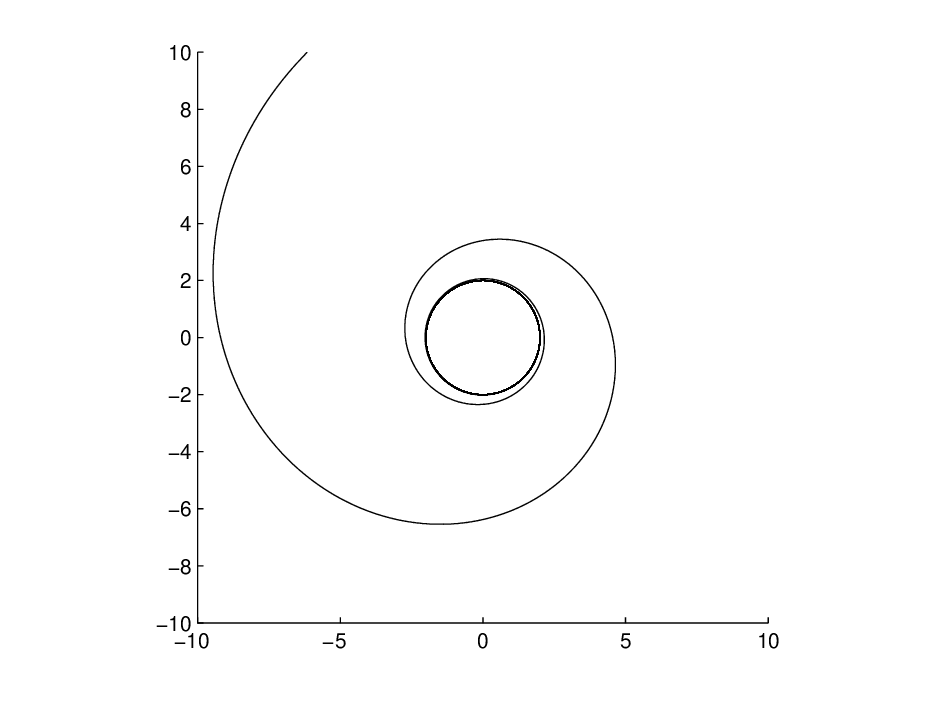} 
\caption{$A = 1$, $B = -0.25$ (rotates and shrinks) Is the comet spiral still embedded here?} 
\label{A1Bm025halastjarnan} 
\end{figure}

\begin{figure}
\centering 
\includegraphics[totalheight=3.4in]{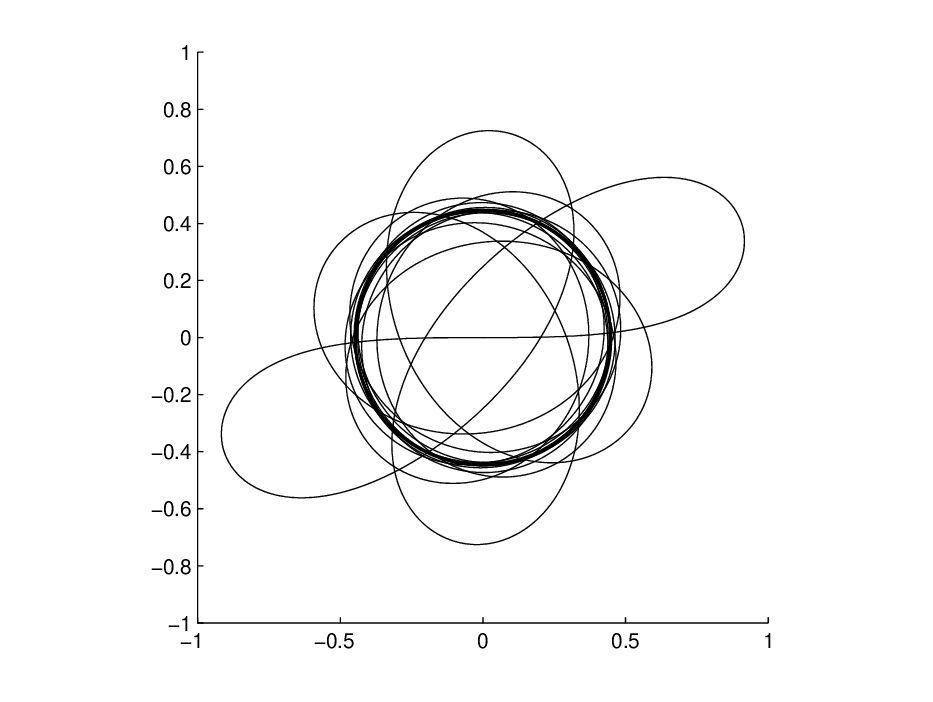} 
\caption{$A = 1$, $B = -5$ (rotates and shrinks) The symmetric curve is no longer embedded when $B$ is small.} 
\label{A1Bm1samhverfi} 
\end{figure}

\begin{figure}
\centering 
\includegraphics[totalheight=3.4in]{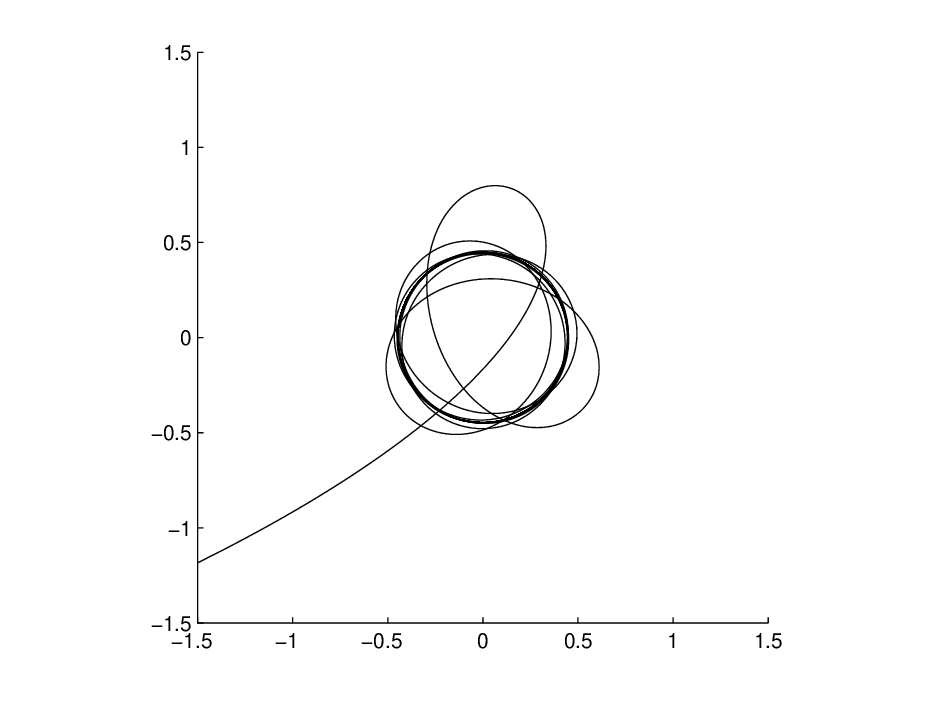} 
\caption{$A = 1$, $B = -5$ (rotates and shrinks) Similarly, the comet spiral is not embedded when $B$ is small.} 
\label{A1Bm5halastjarnan} 
\end{figure}

\clearpage

\section*{Acknowledgements}

The author would like to thank William Minicozzi II for suggesting the problem to him and Caleb Hussey for fruitful discussions on the topic. Moreover, he would like to thank Lu Wang for reading the draft and providing valuable feedback and the referee for suggesting a better proof of Lemma \ref{2ertil}. Finally, he would like to thank his advisor, Tobias Colding, for continuous guidance and support.

\providecommand{\bysame}{\leavevmode\hbox to3em{\hrulefill}\thinspace}
\providecommand{\MR}{\relax\ifhmode\unskip\space\fi MR }
\providecommand{\MRhref}[2]{%
  \href{http://www.ams.org/mathscinet-getitem?mr=#1}{#2}
}
\providecommand{\href}[2]{#2}

\end{document}